\documentclass[11pt]{article}
\usepackage{fullpage}

\usepackage{amsmath,amssymb,amsthm,graphicx}
\usepackage[all]{xy}

\newtheorem{thm}{Theorem}[subsection]

\newtheorem{cor}[thm]{Corollary}
\newtheorem{propn}[thm]{Proposition}
\newtheorem{rem}[thm]{Remark}
\newtheorem{defn}[thm]{Definition}
\newtheorem{lemma}[thm]{Lemma}

\numberwithin{equation}{section}

\def\a{\alpha}

\def\k{\kappa}

\def\*{\bullet}

\newcommand{\surj}{\twoheadrightarrow}

\newcommand{\isomarrow}{\stackrel{\sim}{\rightarrow}}
\newcommand{\darrow}{\stackrel{d}{\rightarrow}}
\newcommand{\lrarrow}{\xymatrix{ \ar@<.7ex>[r] & \ar@<.7ex>[l]}}

\begin{document}
\title{Modules Over a Chiral Algebra}
\author{N. Rozenblyum}
\date{}

\maketitle

\section{Introduction}
The theory of chiral algebras introduced by Beilinson and Drinfeld in \cite{CHA} has found numerous applications in the geometric Langlands program as well as the study of conformal field theory in two dimensions.  Chiral algebras generalize the notion of a vertex algebra and that of an $E_2$-algebra in topology.\\
\\
As explained in
\cite{CHA}, we can regard a chiral algebra $\mathcal{A}$ on a curve $X$ in three different ways:
\begin{enumerate}
\item
A $D$-module on $X$ with a chiral bracket $\{\ ,\ \} :
j_*j^*(\mathcal{A}\boxtimes \mathcal{A}) \rightarrow \Delta_*(A)$.
\item
A Lie algebra in the tensor category of $D$-modules on $Ran(X)$ given
by the chiral tensor product $\otimes^{ch}$.
\item
As a factorization algebra on $Ran(X)$, i.e. a sequence of
quasi-coherent sheaves $\mathcal{A}^{(n)}$ for every power of the
curve $X^n$ satisfying the factorization property.  This description
is Koszul dual to the one above.
\end{enumerate}

Given a chiral algebra $\mathcal{A}$, it is natural to consider modules over it.  For instance, to every semi-simple Lie algebra $\mathfrak{g}$ and an $ad$-invariant bilinear form $\k$, one can associate a chiral algebra $\mathcal{A}_{\mathfrak{g},\kappa}$ on any smooth curve $X$.  Fixing a point $x\in X$, the category of chiral $\mathcal{A}$-modules supported at $x$ is exactly the category of representations of the affine Kac-Moody Lie algebra $\hat{\mathfrak{g}}_{\kappa}$.

From the point of view of conformal field theory, the category of chiral modules supported at a point $x$ is the category assigned to the boundary circle of an infinitesimal disk around $x$.  In order to understand how these categories behave as the point moves, one is led to consider chiral modules on the entire curve.  Such modules are described in \cite{CHA} in both chiral and factorization terms.  Gaitsgory has generalized the notion of factorization modules over a chiral algebra $\mathcal{A}$ to any power of a curve $X$.  He and Lurie consider the concept of a chiral category -- a motivating example of which is the category of modules over a chiral algebra on all powers of the curve.  In their upcoming work, they show that a chiral category with certain holonomicity conditions is equivalent to a braided monoidal category.  In this way, chiral categories are a generalization of braided monoidal categories.  This explains, for instance, the braided monoidal structure on the category $\hat{\mathfrak{g}}_{\kappa}$-modules.

The main goal of this paper is to give a description of modules over a chiral algebra on an arbitrary power of a curve $X$ similarly to the three approaches to chiral algebras above.  In addition, we study a number of related notions and constructions involving Lie-* algebras and factorization spaces.

The paper is organized as follows.  In section \ref{conv}, we fix notation and conventions.  In section \ref{chiral}, we define the parametrized Ran space and give two equivalent definitions of modules over a chiral algebra on an arbitrary power of a curve $X$: in terms of factorization and certain chiral operations.  In section \ref{star}, we consider chiral algebras which are chiral envelopes of Lie-* algebras.  We recall the construction of the chiral envelope and define Lie-* and chiral modules for a Lie-* algebra on arbitrary powers of a curve.

In section \ref{spaces}, we consider factorization spaces, which are nonlinear analogues of chiral algebras.  In the case of counital factorization algebraic stacks, we give a characterization in terms of multijets of a $D_X$ algebraic stack.  We then consider the case of $D_X$ group schemes and construct a functor from modules over a $D_X$ group scheme to Lie-* modules over a corresponding Lie-* algebra which is fully faithful in the case that the group is connected.  In section \ref{tensor}, we consider the chiral and $*$-tensor structures on sheaves on the Ran space and show that sheaves on the parametrized Ran space form a module category with respect to both tensor structures.  We use this to characterize Lie-* and chiral modules as Lie modules in these module categories (with respect to the corresponding tensor structure).

\paragraph{Acknowledgements.}
I would like to thank Jacob Lurie for numerous helpful conversations.  I am deeply grateful to Dennis Gaitsgory for posing the problem addressed in this paper as well as detailed suggestions and encouragement during the various stages of this paper, which have substantially improved it in both substance and exposition.  This work was carried out while the author was supported by an NSF Graduate Fellowship.

\section{Notation and Conventions}\label{conv}
Throughout this paper, $X$ will be an algebraic curve over a field of characteristic 0.  We will consider chiral algebras on $X$ and modules on powers of $X$.  For our purposes, chiral algebras are assumed to be unital.  In the
nonunital case, modules over a chiral algebra are defined as modules
over the chiral algebra with the unit adjoined.\\
\\
Let $I_0$ be a finite set. In what follows, we will consider $D \boxtimes \mathcal{O}_{X^{I_0}}$ modules on the 
schemes of the form $Y\times X^{I_0}$ and their subspaces.  For such a scheme $Z$, we will denote by $(D\boxtimes \mathcal{O})_Z$ the corresponding sheaf of algebras on $Z$. Unless otherwise noted, all modules will be right modules.  Let $\mathcal{M}(Z)$ denote the category of $(D\boxtimes \mathcal{O})$-modules, and let $D\mathcal{M}(Z)$ be its derived stable $\infty$-category \cite{DAGI}.\\
\\
The category $D\mathcal{M}(Z)$ has a natural t-structure with the heart given by right $(D\boxtimes\mathcal{O})$-modules.  In the case that $Z$ is smooth, there is also a t-structure whose heart is given by left $(D\boxtimes\mathcal{O})$-modules.  We will work with the right t-structure.  We have the (right) forgetful functor to the derived $\infty$-category of quasi-coherent sheaves
\[ D\mathcal{M}(Z) \rightarrow DCoh(Z) .\]

Let $f: Z_1\rightarrow Z_2$ be a map.  We then have the functors
\[ f_\* : D\mathcal{M}(Z_1) \lrarrow D\mathcal{M}(Z_2) : f^! \]
The functors $f_\*$ and $f^!$ are not in general adjoint functors.  They behave in the same way as the corresponding ones for $D$-modules. For instance, if $f$ is quasi-finite, then $f_\*$ is left exact; and if $f$ is affine, then $f_\*$ is right exact. If $f$ is proper, then $f_\*$ is left adjoint to $f^!$.    Furthermore, the following diagram commutes
\[ \xymatrix{ D\mathcal{M}(Z_2)\ar[d]\ar[r]^{f^!} & D\mathcal{M}(Z_1)\ar[d] \\
DCoh(Z_2)\ar[r]^{f^!} & DCoh(Z_1)} .\]

\begin{rem}
For a smooth $X^{I_0}$-scheme $Z$, we have a notion of $D$-modules on $Z$ vertical along the projection to $X^{I_0}$.  In the case of a product $Y\times X^{I_0}$, $D\boxtimes \mathcal{O}_{X^{I_0}}$ modules on $Y\times X^{I_0}$ are the same as $D$-modules vertical along the projection.  As in the upcoming work of Gaitsgory and Lurie, the category of $D$-modules vertical along the projection naturally has a connection along the base, and the $D$-objects with respect to this connection recover $D$-modules on $Z$.  All our constructions will be compatible with this connection. Therefore, we will work with $D\boxtimes \mathcal{O}$-modules and the analogous results will be valid for $D$-modules as well by passing to $D$-objects.  In particular, we will define and work with chiral modules which on $X^{I_0}$ which are quasi-coherent sheaves on $X^{I_0}$.  Passing to $D$-objects in this category recovers the chiral modules which are $D$-modules.
\end{rem}

\section{Chiral Modules}\label{chiral}
\subsection{Subspace of Diagonals}
Let $X$ be a curve and let $I_0$ be a fixed finite set.  For a finite set $I$ with an embedding $\Phi: I_0\hookrightarrow I$, consider the $|I_0|$-dimensional closed subscheme $H_{\Phi} \subset X^I$ given by the union of the diagonal subschemes
\begin{equation} H_{\Phi} = \bigcup_{\pi:I_1\rightarrow I_0} \{x_i = y_{\pi(i)}\ \mbox{ for } i\in I_1 \} \subset X^{I_1}\times X^{I_0}
\end{equation}
where $I=I_0\sqcup I_1$ and $x_i$ and $y_j$ are the coordinates on $X^{I_1}$ and $X^{I_0}$ respectively.  Equivalently, the union is over all surjections $\pi: I \surj I_0$ which are the identity on $I_0$.\\
\\
The subspaces $H_{\Phi}$ have the following factorization property.  Suppose we are given a partition $\alpha: I_0 = I_0^{(1)} \sqcup I_0^{(2)}$. Let $U_\alpha = \{y_{j_1}\neq y_{j_2} \mbox{ for } j_i\in I_0^{(i)}\} \subset X^{I_1} \times X^{I_0}$.  We then have
\begin{equation}\label{gammafact}
H_{\Phi} \cap U_\alpha = \bigsqcup_{\Phi = \Phi_1 \sqcup \Phi_2} \left(H_{\Phi_1} \times H_{\Phi_2}\right) \cap U_\alpha
\end{equation}
where the union is over partitions $I = I^{(1)} \sqcup I^{(2)}$ such that $\Phi(I_0^{(i)})\subset I^{(i)}$ and $\Phi_i = \Phi|_{I_0^{(i)}}$ for $i=1,2$.\\
\\
Let $M$ be a quasi-coherent sheaf on $X^{I_0}$.  Consider the following diagram
\begin{equation}
\xymatrix{  H_{\Phi} \ar[r]^-{i} & X^{I_1}\times X^{I_0} \ar[dl]_{p_1}\ar[dr]^{p_2} \\
X^{I_1} & & X^{I_0} }
\end{equation}
Now, let $\Gamma_{\Phi}(M)$ be the $D_{X^{I_1}}\boxtimes \mathcal{O}_{X^{I_0}}$-module given by
\begin{equation}
 \Gamma_{\Phi}(M) := i_\*i^! p_2^!(M) .
\end{equation}
Note that we have the exact sequence
\begin{equation}
0 \rightarrow p_2^!(M)[-I_1] \rightarrow j_\*j^!(p_2^!(M))[-I_1] \rightarrow \Gamma_{\Phi}(M)\rightarrow 0.
\end{equation}
The functor $\Gamma_{\Phi}$ has the following transitivity property.
\begin{lemma}
Let $M$ be a quasicoherent sheaf on $X^{I_0}$.  We then have for $I_0\subset I$ and $I\subset J$,
\[ \Gamma_{(I\subset J)}\circ \Gamma_{(I_0\subset I)} (M) = \Gamma_{(I_0\subset J)} (M) .\]
\end{lemma}
\begin{proof}
Consider the diagram
\[ \xymatrix{ H_{(I_0\subset J)} \ar[r]^-{i_2}\ar[d]_{\pi_2} & H_{(I\subset J)} \ar[r]^-{i_3}\ar[d]_{\pi_3} & X^J \ar[dl]\\
H_{(I_0\subset I)} \ar[r]^-{i_1}\ar[d]_{\pi_1} & X^I \ar[dl]_{p_1} \\
X^{I_0} } \]
The square is Cartesian.  We have $\Gamma_{(I_0\subset I)}(M) = i_{1\*}\pi_1^!(M)$ and $\Gamma_{(I\subset J)}\circ \Gamma_{(I_0\subset I)}(M) = i_{3\*}\pi_3^!(\Gamma_{(I_0\subset I)}(M))$.  Since
$\pi_3^! i_{1 \*} = i_{2 \*} \pi_2^!$, we have that
\[ i_{3 \*} \pi_3^!(\Gamma_{(I_0\subset I)}(M)) = i_{3 \*} i_{2 \*} \pi_2^! \pi_1^!(M) = \Gamma_{(I_0\subset J)}(M). \]
\end{proof}

We also have the following factorization property for $\Gamma$, which follows directly from (\ref{gammafact}).
\begin{lemma}\label{gfact}
Let $M$ and $N$ be quasi-coherent sheaves on $X^{I}$ and $X^{J}$ respectively. Let $K=I\sqcup J$.  We then have a natural isomorphism
\[ j_{(x_i\neq y_j) \*} j_{(x_i\neq y_j)}^!(\Gamma_{(K \subset K \sqcup [1])} (M\boxtimes N)) \simeq
j_{(x_i\neq y_j) \*} j_{(x_i\neq y_j)}^!\left(
\begin{array}{c}
 j_{(z\neq y_j) \*} j_{(z\neq y_j)}^!( \Gamma_{(I\subset I\sqcup [1])}(M) \boxtimes N)\\
  \oplus\\
   j_{(z\neq x_i) \*} j_{(z\neq x_i)}^!(M\boxtimes \Gamma_{(J\subset J\sqcup [1])}(N))
\end{array}
  \right) \]
where $x_i, y_j$ and $z$ are the coordinates on $X^{I}\times X^{J}\times X$.
\end{lemma}

\subsection{Parametrized Ran Space}
Let $I_0$ be a finite set.  To keep track of the combinatorics of various powers of $X$ over $X^{I_0}$, it will be convenient to use the following category.  Let $\mathcal{S}_{I_0}$ be the category whose objects are finite sets with an embedding $I_0\subset I$ of $I_0$, and the morphisms $Hom_{\mathcal{S}_{I_0}}(I_0\subset I, I_0\subset J)$ are given by surjective maps $I\surj J$ which restrict to the identity on $I_0$.  For a finite set $I$, let $\tilde{I}:=I\sqcup I_0$.  This gives a functor from the category of finite sets with surjections to $\mathcal{S}_{I_0}$ given by $I \mapsto (I_0 \subset \tilde{I})$ which is essentially surjective.

Now, we can define the parametrized Ran space $Ran_{I_0}(X)$,
which is an $\mathcal{S}_{I_0}^{op}$ diagram of schemes.  To each
$(I_0\subset I)\in \mathcal{S}_{I_0}$, we assign the product $X^{I}$,
and for each morphism $\pi: J \surj I$ in $\mathcal{S}_{I_0}$, we have the
corresponding diagonal map
\[ \Delta^{(\pi)}: X^{I}\rightarrow X^{J} .\]
Morally, we regard $Ran_{I_0}(X)$ as the colimit of this diagram; namely, we consider $Ran_{I_0}(X)$ as the space of nonempty finite subsets of $X$ with an ordered subset of (possibly repeating) $I_0$ points. While there is no such colimit in the category of algebraic spaces, we can reasonably consider sheaves on $Ran_{I_0}(X)$.

Let $D\mathcal{M}(X^{\tilde{I}})$ be the stable $\infty$-category of $D_{X^I}\boxtimes \mathcal{O}_{X^{I_0}}$ modules for $(I_0\subset \tilde{I})\in \mathcal{S}_{I_0}$.  For each $\pi: \tilde{J} \surj \tilde{I}$, we have the derived functor
\begin{equation}\label{DMfunct}
 \Delta^{(\pi)!}: D\mathcal{M}(X^{\tilde{J}}) \rightarrow D\mathcal{M}(X^{\tilde{I}}).
\end{equation}
We define $D\mathcal{M}(Ran_{I_0})$ to be the homotopy limit of this diagram.  Roughly, an object $M$ of $D\mathcal{M}(Ran_{I_0})$ consists of a collection of $D_{X^I}\boxtimes \mathcal{O}_{X^{I_0}}$ modules $M^{\tilde{I}}$ along with isomorphisms
\[ \Delta^{(\pi)!}(M^{\tilde{J}}) \isomarrow M^{\tilde{I}} \]
and suitable compatibility data.
More precisely, the functors (\ref{DMfunct}) give a Cartesian fibration
 $p: D\mathcal{M}^{!}_{I_0} \rightarrow \mathcal{S}_{I_0}$ \cite{Topos} and we define $D\mathcal{M}(Ran_{I_0})$ to be the category of Cartesian sections of $p$.  $D\mathcal{M}(Ran_{I_0})$ is a stable $\infty$-category.\\
 \\
 In what follows it will also be convenient to consider the lax limit of this diagram.  For each morphism, $\pi: \tilde{J}\surj \tilde{I}$ in $\mathcal{S}_{I_0}$, we have the functors
 \[ \Delta^{(\pi)}_\*: D\mathcal{M}(X^{\tilde{I}}) \rightarrow D\mathcal{M}(X^{\tilde{J}})
 \]
 which give a coCartesian fibration $q: D\mathcal{M}_{I_0\*} \rightarrow \mathcal{S}_{I_0}$.  Let $D\mathcal{M}(X^{\mathcal{S}_{I_0}})$ be the category of sections of $q$.  Roughly, an object of $D\mathcal{M}(X^{\mathcal{S}_{I_0}})$ is a collection of $D_{X^I}\boxtimes \mathcal{O}_{X^{I_0}}$ modules $M^{\tilde{I}}$ along with morphisms
\[ \Delta^{(\pi)}_\*(M^{\tilde{I}}) \rightarrow M^{\tilde{J}} \]
and suitable compatibility data.  By adjunction, we have the inclusion functor
\[ r^!: D\mathcal{M}(Ran_{I_0})\rightarrow D\mathcal{M}(X^{\mathcal{S}_{I_0}}) \]
which has a left adjoint
\[ r_*: D\mathcal{M}(X^{\mathcal{S}_{I_0}}) \rightarrow D\mathcal{M}(Ran_{I_0}). \]
These realize the category $D\mathcal{M}(Ran_{I_0})$ as a localization of $D\mathcal{M}(X^{\mathcal{S}_{I_0}})$.\\
\\
By construction, we have the functors
\[ s_{\tilde{I}*}: D\mathcal{M}(X^{\tilde{I}}) \rightarrow D\mathcal{M}(X^{\mathcal{S}_{I_0}}) \]
for each $\tilde{I}$ given by pushforward along with a right adjoint
\[ s^{!}_{\tilde{I}}: D\mathcal{M}(X^{\mathcal{S}_{I_0}}) \rightarrow D\mathcal{M}(X^{\tilde{I}}) \]
given by $\{M^{\tilde{J}}\} \mapsto M^{\tilde{I}}$.  We thus have a diagram of adjoint functors
\begin{equation}
D\mathcal{M}(X^{\tilde{I}}) \overset{s_{\tilde{I}*}}{\underset{s^!_{\tilde{I}}}{\lrarrow}} D\mathcal{M}(X^{\mathcal{S}_{I_0}}) \overset{r_*}{\underset{r^!}{\lrarrow}} D\mathcal{M}(Ran_{I_0}) .
\end{equation}
Explicitly, the left adjoints are given by
\[ s_{\tilde{I}*}(M)_{X^{\tilde{J}}} = \underset{\pi:\tilde{J}\surj \tilde{I}}{\bigoplus} \Delta^{(\pi)}_\* (M) \ \ \ \mbox{ for } M\in D\mathcal{M}(X^{\tilde{I}}) \mbox{ and} \]
\[ r_*(N)_{X^{\tilde{J}}} = \underset{\pi: \tilde{J}\surj \tilde{I}}{\mbox{hocolim }} \Delta^{(\pi)!} N_{X^{\tilde{I}}}\ \ \ \mbox{ for } N\in D\mathcal{M}(X^{\mathcal{S}_k}).\]
Let $r_{\tilde{I}*} = s_{\tilde{I}*} \circ r_*$.
We have that
\[ r_{I_0*}(M)_{X^{\tilde{J}}} = \Gamma_{(I_0\subset \tilde{J})}(M)  \ \ \ \mbox{ for } M\in D\mathcal{M}(X^{I_0}).\]
In particular, $r_{I_0*}(M)_{X^{I_0}} \simeq M$ and therefore the functor
\[ r_{I_0*} : D\mathcal{M}(X^{I_0})\rightarrow D\mathcal{M}(Ran_{I_0}) \]
is a fully faithful embedding.\\
\\
For $\tilde{I}\neq I_0$ we can describe this functor as follows.  For finite sets $I,J$ let
\[ H_{I, J; I_0} \subset X^I\times X^J \times X^{I_0} \]
be the incidence correspondence given by $\{x_i, z_k\} = \{y_i, z_k\}$ (as sets) where $x_i, y_j$ and $z_k$ are the coordinates of $X^I$, $X^J$ and $X^{I_0}$ respectively.  We have the correspondence
\[ \xymatrix{ & \ar[dl]_{\pi_1} H_{I, J; I_0}\ar[dr]^{\pi_2} &\\ X^{\tilde{I}} && X^{\tilde{J}} } \]
and $r_{\tilde{I}*}$ is given by
\[ r_{\tilde{I}*}(M)_{X^{\tilde{J}}} = \pi_{2\*}\pi_1^!(M)  \ \ \ \mbox{ for } M\in D\mathcal{M}(X^{\tilde{I}}).\]

The category $D\mathcal{M}(Ran_{I_0})$ has two natural $t$-structures.  In the projective $t$-structure, an object $M$ is co-connective if each $r_{\tilde{I}}^!(M)[-\tilde{I}]$ is co-connective.  The connective objects are generated by $r_{\tilde{I}*}(M_I)[\tilde{I}]$ for $M_I \in D\mathcal{M}(X^{\tilde{I}})$ connective.  Similarly, in the inductive $t$-structure, the connective objects are generated by $r_{\tilde{I}*}(M_I)$.\\
\\
We will consider the subcategory $\mathcal{M}(Ran_{I_0}) \subset D\mathcal{M}(Ran_{I_0})$ which is not abelian but is a subcategory of the abelian category given by the heart of the projective $t$-structure.  An object of $\mathcal{M}(Ran_{I_0})$ is given by a collection of $D_{X^I}\boxtimes \mathcal{O}_{X^{I_0}}$ modules $M^{\tilde{I}}$ along with compatible derived isomorphisms
\[ \nu_F^{(\pi)}: \Delta^{(\pi)!} M^{\tilde{J}}[|J|-|I|] \isomarrow \mathcal M^{\tilde{I}} \]
for each morphism $\pi: \tilde{J}\surj \tilde{I}$ in $\mathcal{S}_{I_0}$.  In particular, we require that each $M^{\tilde{I}}$ have no local sections on the diagonals.
The inclusion $\mathcal{M}(Ran_{I_0}) \rightarrow D\mathcal{M}(Ran_{I_0})$ is given by $\{M^{\tilde{I}}\} \mapsto \{M^{\tilde{I}}[\tilde{I}]\}$.  In what follows we will mostly be dealing with this category rather than the derived version though pretty much everything is also true in the derived setting.\\
\\
As in the derived case, we have adjoint functors
\[
\mathcal{M}(X^{I_0}) \overset{r_{I_0*}}{\underset{r^!_{I_0}}{\lrarrow}} \mathcal{M}(Ran_{I_0})
\]
with ${r_{I_0*}}$ a fully faithful embedding.  For other $\tilde{I}$, the functors $r_{\tilde{I}*}: D\mathcal{M}(X^{\tilde{I}})\rightarrow D\mathcal{M}(Ran_{I_0})$ don't map $M\in\mathcal{M}(X^{\tilde{I}})$ to an object in $\mathcal{M}(Ran_{I_0})$ unless $M$ has no local sections supported on the diagonals.

\subsection{Factorization Modules}
For $(I_0\subset J)\in \mathcal{S}_{I_0}$ and a pointed finite set $\star\in I$, consider the set of partitions
\[ Part_{I,\star}(J,I_0): = \{\mbox{partitions } J=\sqcup_{i\in I} J_i \mbox{ such that } I_0\subset J_{\star}\} .\]
Note that we can naturally identify
\[ Part_{I,\star}(J,I_0) \simeq \{ J \surj I \mbox{ such that } I_0\mapsto \star \} .\]
Namely, given such a surjection $\pi: J\surj I$, let $J_i = \pi^{-1}(i)$ (and vice versa).
In what follows, it will often be convenient to describe partitions as such surjective maps.  For a finite set $I$, let $I^+$ denote the pointed set $I\sqcup \{\star\}$.

Note that a map $J\surj I$ in $\mathcal{S}_{I_0}$ defines a partition of $I$ by composing with the collapsing map: $J\surj I\surj I/I_0$.  In particular for a map $\tilde{J}\surj \tilde{I}$, the corresponding partition is given by $\tilde{J}\surj I^{+}$.

Let $\pi: J\surj I$ be a partition.  Consider the open subset
\[ j^{(J/I)}: U^{(J/I)} \hookrightarrow X^J \]
given by $U^{(J/I)} = \{ (x_i)\in X^J \ |\  x_i\neq x_j \mbox{ if } \pi(i)\neq \pi(j) \}$.  Slightly abusing notation, for a map $\pi: J\surj I$ in $\mathcal{S}_{I_0}$, we will denote by $j^{(J/I)}: U^{(J/I)} \hookrightarrow X^I$ the open set corresponding to the partition obtained by collapsing $I_0$ in $I$.  Thus, for such a map, we have $J_i = \pi^{-1}(i)$ for $i\in I-I_0$ and $J_{\star} = \pi^{-1}(I_0)$.
  We will also denote by $j^{(J)}: U^{(J)}\hookrightarrow X^J$ the open set corresponding to the partition which consists of $J_1 = J - I_0$ and $I_0$.

\begin{defn}
Let $\mathcal{A}$ be a chiral algebra over a curve $X$, $I_0$ a finite set.
A factorization $\mathcal{A}$-module $M$ on $X^{I_0}$ is a $D\boxtimes \mathcal{O}_{X^{I_0}}$-module on
$Ran_{I_0}(X)$ along with the data of an isomorphism
\[ c_{[J/I]}: j^{(J/I)!}(\boxtimes_{i\in I-I_0} \mathcal{A}^{(J_i)} \boxtimes M^{(J_{\star})})\isomarrow j^{(J/I)!} M^{(J)}\]
for each $\pi: J\surj I$ in $\mathcal{S}_{I_0}$ satisfying
the following conditions.
\begin{enumerate}
\item
The unit map $\omega_{X^I} \rightarrow \mathcal{A}^{(I)}$ induces a map of
sheaves on $Ran_{I_0}$:
\[ \omega_{X^{I}}\boxtimes M \rightarrow M^{(I)} \]
which is compatible with factorization isomorphisms and is the
identity on $X^{I_0}$.
\item
For each $J\surj J'\surj I$ in $\mathcal{S}_{I_0}$ we have
\[ c_{[J/J']} =
c_{[J/I]}\left(\underset{i\in I-I_0}{\boxtimes} c_{[J_i/J'_i]} \boxtimes
c_{[J_{\star}/J'_{\star}]}\right) \]
(where $c_{[J_i/J'_i]}$ are
the factorization isomorphisms for $\mathcal{A}$) and the following
diagram commutes
\[ \xymatrix{
\Delta^{(J/J')!} j^{(J/I)!}
(\boxtimes_{i\in I-I_0} \mathcal{A}^{(J_i)} \boxtimes
M^{(J_{\star})})[|J|-|J']] \ar[r]^-{c_{[I/J]}}
\ar[d]_-{\boxtimes_{i\in I-I_0}
  (\nu^{J_i/J'_i})\ \boxtimes\ \nu^{J_{\star}/J'_{\star}}
} & \Delta^{(J/J')!} j^{(J/I)!}
M^{(J)}[|J|-|J'|] \ar[d]^-{\nu^{(J/J')}}
\\ j^{(J'/I)!} ( \boxtimes_{i\in I-I_0}
\mathcal{A}^{(J'_i)} \boxtimes M^{(J'_{\star})})
\ar[r]^-{c_{[J'/I]}} &
j^{(J'/I)!} M^{(J')} } \]
\end{enumerate}
\end{defn}

\begin{rem}
\begin{enumerate}
\item
If $M$ is a factorization $\mathcal{A}$ module on $X^{I_0}$
then each $M^{J_0}$ for $I_0\subset J_0$ is also a chiral
$\mathcal{A}$ module on $X^{J_0}$.
\item
Suppose $M$ and $N$ are factorization $\mathcal{A}$ modules on $X^{I_0}$ and
$X^{J_0}$ respectively.  Then the sheaf $j_\*j^!(M\boxtimes N)$, where $j:
U\rightarrow X^{I_0\sqcup J_0}$ is the open subset on which the first $I_0$
components are distinct from the last $J_0$, is a chiral $\mathcal{A}$
module on $X^{I_0 \sqcup J_0}$.  We should think of this module as representing
the tensor product of $M$ and $N$ in the chiral category of
$\mathcal{A}$ modules.
\item
Given a surjective map of finite sets $J_0\surj I_0$, we have inclusion maps \[ T_I: X^{I} \rightarrow X^{J} \]
for $(I_0\subset I) \in \mathcal{S}_{I_0}$ and $J = I\cup_{I_0} J_0$.
For a factorization module $\{M^{(I)}\}$ on $X^{I_0}$, the collection of $(D\boxtimes \mathcal{O}_{J_0})$ modules given by $\{T_{I\bullet}(M^{(I)})\}$ form a factorization module on $X^{J_0}$.  This defines the direct image functor from factorization modules on $X^{I_0}$ to those on $X^{J_0}$.
\end{enumerate}
\end{rem}

\subsection{Chiral Modules}
We can now define a chiral module $M$ on $X^{I_0}$
over a chiral algebra $\mathcal{A}$ in terms of certain chiral
operations.

\begin{defn}
Let $\mathcal{A}$ be a chiral algebra, $I_0$ a finite set.  A chiral $\mathcal{A}$ module on $X^{I_0}$ is a quasi-coherent sheaf $M$ on $X^{I_0}$ along with a map
\[ \mu^{I_0}: j_\* j^! (\mathcal{A}\boxtimes M) \rightarrow \Gamma_{(I_0\subset \tilde{[1]})}(M) \]
such that the following conditions are satisfied.
\begin{enumerate}
\item {\bf (Unit)}
The following diagram commutes:
\begin{equation}\label{unit}
\xymatrix{ j_\*j^!(\omega \boxtimes M) \ar[r]\ar[d] &
  j_\*j^!(\mathcal{A}\boxtimes M)\ar[d] \\ \Gamma_{(I_0\subset \tilde{[1]})}(M) \ar@{=}[r] & \Gamma_{(I_0\subset \tilde{[1]})}(M) }
\end{equation}
\item {\bf (Lie action)}
Note that we have a canonical map
\[ j^{(\tilde{[2]})}_\* j^{(\tilde{[2]}) !} (\mathcal{A}\boxtimes \Gamma_{(I_0\subset \tilde{[1]})}(M)) \rightarrow \Gamma_{(\tilde{[1]}\subset \tilde{[2]})}(j_\* j^!(\mathcal{A}\boxtimes M)) .\]
Composing this with $\Gamma_{(I_0\subset \tilde{[1]})}$ applied to the action map, we obtain a map
\[ j_\* j^!(\mathcal{A}\boxtimes \Gamma_{(I_0\subset \tilde{[1]})}(M)) \rightarrow \Gamma_{(I_0\subset \tilde{[2]})}(M) .\]
Thus we can compose the action map in several ways:
\begin{align*}
\mu_{1,\{2,3\}}&: j_\*j^!(\mathcal{A}\boxtimes \mathcal{A}\boxtimes M) \cong j_\*j^!(\mathcal{A}\boxtimes j_\*j^!(\mathcal{A}\boxtimes M)) \rightarrow j_\*j^!(\mathcal{A}\boxtimes \Gamma_{(I_0\subset \tilde{[1]})}(M))\rightarrow \Gamma_{(I_0\subset \tilde{[2]})}(M) \\
\mu_{2,\{1,3\}} &= \mu_{1,\{2,3\}} \circ \sigma_{12}^!:  j_\*j^!(\mathcal{A}\boxtimes \mathcal{A}\boxtimes M) \rightarrow \Gamma_{(I_0\subset \tilde{[2]})}(M) \\
\mu_{\{1,2\},3} &: j_\* j^!(\mathcal{A}\boxtimes\mathcal{A}\boxtimes M) \cong j_\* j^!( j_\* j^!(\mathcal{A}\boxtimes \mathcal{A}) \boxtimes M) \rightarrow j_\*j^!(\Delta_\*(\mathcal{A}\boxtimes M)) \cong \\
& \cong \Delta_{12\*}(j_\*j^!(\mathcal{A}\boxtimes M)) \rightarrow \Delta_{12\*} \Gamma_{(I_0\subset \tilde{[1]})}(M) \rightarrow \Gamma_{(I_0\subset \tilde{[2]})}(M).
\end{align*}
We demand that
\begin{equation}\label{lieact}
\mu_{\{1,2\},3} = \mu_{2,\{1,3\}} - \mu_{1,\{2,3\}} .
\end{equation}
\end{enumerate}
\end{defn}

\subsection{Equivalence of Categories}
Our main goal in this section is to prove the following result showing that the notions of a factorization and chiral module agree:

\begin{thm}\label{equiv} Let $\mathcal{A}$ be a chiral algebra.  As categories over quasi-coherent sheaves on $X^{I_0}$, the categories of chiral $\mathcal{A}$-modules and factorization $\mathcal{A}$-modules are equivalent.
\end{thm}

First, let's consider a decomposition of a factorization module.  Let $M$ be a factorization module on $X^{I_0}$ for a chiral algebra $\mathcal{A}$.  On $X\times X^{I_0}$ consider the Cousin complex for the sheaf
$M^{\tilde{[1]}}$ with respect to the stratification given by
\[ i: H_{(I_0\subset \tilde{[1]})} \hookrightarrow X\times X^{I_0} \hookleftarrow U^{\tilde{([1])}} : j \]
Since $M^{\tilde{[1]}}$ does not have local sections along $H$, we have an exact sequence
\begin{equation}\label{cousin} 0 \rightarrow M^{\tilde{[1]}} \rightarrow j_\*j^!(\mathcal{A}\boxtimes M) \rightarrow i_{\*}i^{!} M^{\tilde{[1]}}[1] \rightarrow 0 .
\end{equation}

\begin{lemma}\label{Gamma} Let $M$ be a factorization $\mathcal{A}$-module on $X^{I_0}$.  Then there is a natural isomorphism $i_{\*}i^! M^{\tilde{[1]}}[1] \simeq \Gamma_{(I_0\subset \tilde{[1]})}(M)$.
\end{lemma}
\begin{proof}
The unit map gives a map $\omega_X\boxtimes M \rightarrow M^{\tilde{[1]}}$ which is compatible with factorization.  We thus have a commutative diagram
\[ \xymatrix{j_\*j^!(\omega\boxtimes M) \ar[r]\ar[d] & j_\* j^!(M^{\tilde{[1]}}) \ar[d] \\ \Gamma_{(I_0\subset \tilde{[1]})}(M) \ar[r] & i_\*i^!(M^{\tilde{[1]}})[1]} \]
We wish to show that the bottom map is an isomorphism.  Since both $(D\boxtimes \mathcal{O})$-modules are supported on the union of diagonal embeddings $X^{I_0} \rightarrow X^{\tilde{[1]}}$, it suffices to check that the map is an isomorphism when pulled back to each $X^{I_0}$.
By the unit axiom, these pullbacks give the identity map on $M$.
\end{proof}

\begin{proof}[Proof of Theorem \ref{equiv}]
Let $M$ be a factorization $\mathcal{A}$-module. By (\ref{cousin}) and Lemma \ref{Gamma}, we have that $M^{\tilde{[1]}}$ is quasi-isomorphic to the complex
\begin{equation}\label{chiralmult}
\xymatrix{j_\*j^!(\mathcal{A}\boxtimes M)\ar[r]^-{\mu^{I_0}}& \Gamma_{(I_0\subset \tilde{[1]})}(M)}
\end{equation}
We wish to show that this map satisfied the unit and Lie action axioms.  The unit axiom clearly follows from the unit axiom for factorization modules.\\
\\
In what follows the following notation will be convenient.  Let $I$ be a finite set with $|I|=n$, and let $I^+$ be the pointed set $I\sqcup \{\star\}$ as before.  For a sequence $(\a_1,\ldots,\a_n)$ with $\a_i\in I^+$, let
\[ j_{(\a_1,\ldots, \a_n)}: U_{(\a_1,\ldots, \a_n)} \rightarrow X^{\tilde{I}} \]
be the inclusion of the open set corresponding to the partition $\tilde{I}\surj \{\a_i, \star\}$ given by $i\mapsto \a_i$ for $i\in I$.\\
\\
Let $C(\mathcal{A}, M)^{\tilde{[2]}}$ be the Cousin complex of $M^{\tilde{[2]}}$ with respect to the stratification given by the diagonals:
\begin{equation}\label{c2}
j_{(1,2)\*} j_{(1,2)}^!(\mathcal{A}\boxtimes\mathcal{A}\boxtimes M) \rightarrow 
\begin{array}{c}
j_{(1,\star)\*}j_{(1,\star)}^! (\mathcal{A}\boxtimes \Gamma_{(I_0\subset \tilde{[1]})}(M)) \\ \oplus
\\ j_{(\star, 1)\*}j_{(\star,1)}^!(\mathcal{A}\boxtimes \Gamma_{(I_0\subset \tilde{[1]})}(M)) \\ \oplus \\ 
j_{(1,1)\*}j_{(1,1)}^! \Delta_{12\*} (\mathcal{A}\boxtimes M)
\end{array} \rightarrow \Gamma_{(I_0\subset \tilde{[2]})}(M).
\end{equation}
The fact that this is a complex implies the Lie action identity for $\mu^{I_0}$.
\\

Now, let $M$ be a chiral $\mathcal{A}$-module.  We will construct a factorization module out of it by constructing the sheaves $M^{\tilde{[i]}}$ inductively.  From the identity axiom, we have that $\mu^{I_0}$ is surjective.  Therefore, the complex
\[ \xymatrix{j_\*j^!(\mathcal{A}\boxtimes M)\ar[r]^-{\mu^{I_0}}& \Gamma_{(I_0\subset \tilde{[1]})}(M)} \]
is quasi-isomorphic to a sheaf $M^{\tilde{[1]}}$ on $X^{\tilde{[1]}}$.  By induction, we need to show that $M^{\tilde{[1]}}$ is a chiral $\mathcal{A}$-module on $X^{\tilde{[1]}}$.\\
\\
Let $C(\mathcal{A},M)^{(\tilde{[2]})}$ be defined as in (\ref{c2}).  It is a complex because of the Lie action identity.
Similarly, let $C(\omega,\mathcal{A},M)$ denote the following complex
\begin{equation}
j_{(1,2)\*} j_{(1,2)}^!(\omega\boxtimes\mathcal{A}\boxtimes M) \rightarrow 
\begin{array}{c}
j_{(1,\star)\*}j_{(1,\star)}^! (\omega\boxtimes \Gamma_{(I_0\subset \tilde{[1]})}(M)) \\ \oplus
\\ j_{(\star, 1)\*}j_{(\star,1)}^!(\mathcal{A}\boxtimes \Gamma_{(I_0\subset \tilde{[1]})}(M)) \\ \oplus \\ 
j_{(1,1)\*}j_{(1,1)}^! \Delta_{12\*} (\mathcal{A}\boxtimes M)
\end{array} \rightarrow \Gamma_{(I_0\subset \tilde{[2]})}(M).
\end{equation}
From the unit axiom, we have a map of complexes $C(\omega,\mathcal{A},M)\rightarrow C(\mathcal{A},M)^{(\tilde{[2]})}$.  Observe that we have natural quasi-isomorphisms
\[ \left(j_{(1,2)\*} j_{(1,2)}^!(\omega\boxtimes\mathcal{A}\boxtimes M) \rightarrow j_{(1,\star)\*}j_{(1,\star)}^! (\omega\boxtimes \Gamma_{(I_0\subset \tilde{[1]})}(M)) \right) \simeq j_{(1,\star)\*}j_{(1,\star)}^! (\omega\boxtimes M^{(\tilde{[1]})}) \mbox{ and} \]

\[ \left(j_{(1,2)\*} j_{(1,2)}^!(\mathcal{A}\boxtimes\mathcal{A}\boxtimes M) \rightarrow j_{(1,\star)\*}j_{(1,\star)}^! (\mathcal{A}\boxtimes \Gamma_{(I_0\subset \tilde{[1]})}(M)) \right) \simeq j_{(1,\star)\*}j_{(1,\star)}^! (\mathcal{A}\boxtimes M^{(\tilde{[1]})}).  \]

Furthermore, it follows from Lemma \ref{gfact} that
\[\left(
\begin{array}{c}
j_{(\star, 1)\*}j_{(\star,1)}^!(\mathcal{A}\boxtimes \Gamma_{(I_0\subset \tilde{[1]})}(M)) \\ \oplus \\ 
j_{(1,1)\*}j_{(1,1)}^! \Delta_{12\*} (\mathcal{A}\boxtimes M)
\end{array} \rightarrow \Gamma_{(I_0\subset \tilde{[2]})}(M) \right) \simeq
\Gamma_{(\tilde{[1]}\subset \tilde{[2]})}(M^{(\tilde{[1]})}) .\]
Thus, $C(\mathcal{A},M)^{(\tilde{[2]})}$ is quasi-isomorphic to the complex
\[ j_{(1,\star)\*}j_{(1,\star)}^! (\mathcal{A}\boxtimes M^{(\tilde{[1]})}) \overset{\mu^{\tilde{[1]}}}{\longrightarrow} \Gamma_{(I_0\subset \tilde{[1]})}(M^{(\tilde{[1]})}) \]
which gives us the desired action map for $M^{(\tilde{[1]})}$ satisfying the unit axiom.\\
\\
To show that the Lie action condition is satisfied, we need to show that $C(\mathcal{A},M^{\tilde{[1]}})^{(\tilde{[3]})}$ is a complex.  We have embeddings
\[ \Gamma_{(\tilde{[1]}\subset \tilde{[2]})}(M^{\tilde{[1]}}) \lhook\joinrel\longrightarrow
\begin{array}{c}
 j_{(1,\star) \*}j_{(1,\star)}^!(\mathcal{A}\boxtimes \Gamma_{(I_0\subset \tilde{[1]})}(M))\\
  \oplus \\
 j_{(1,1) \*} j^!_{(1,1)} \Delta_{1,2 \*}(\mathcal{A}\boxtimes M)
 \end{array} \]
and
\[ \Gamma_{(\tilde{[1]} \subset \tilde{[3]})}(M^{\tilde{[1]}}) \lhook\joinrel\longrightarrow 
\begin{array}{c}
j_{(1,1,1)\*}j_{(1,1,1)}^! \Delta_{123\*}(\mathcal{A}\boxtimes M)\\ \oplus \\
 j_{(1,1,\star)\*}j_{(1,1,\star)}^! \Delta_{12\*}( \mathcal{A}\boxtimes M)\\ \oplus \\
 j_{(1,\star,1)\*}j_{(1,\star,1)}^! \Delta_{12\*}( \mathcal{A}\boxtimes M) \\ \oplus \\
  j_{(1,\star,\star)\*}j_{(1,\star,\star)}^! (\mathcal{A}\boxtimes \Gamma_{(I_0\subset \tilde{[2]})}(M))
\end{array} .\]
It follows that $C(\mathcal{A},M^{\tilde{[1]}})^{(\tilde{[3]})}$ embeds into
\[
j_{(1,2,3)\*} j_{(1,2,3)}^!(\mathcal{A}\boxtimes\mathcal{A}\boxtimes\mathcal{A}\boxtimes M) \longrightarrow 
\begin{array}{c}
\oplus_{\a \in P(1,1,2)} j_{\a\*}j_{\a}^! \Delta_{\a\*}(\mathcal{A}\boxtimes \mathcal{A}\boxtimes M) \\ \oplus
\\ \oplus_{\a \in P(\star,1,2)}\  j_{\a\*}j_{\a}^! (\mathcal{A}\boxtimes \mathcal{A}\boxtimes \Gamma_{(I_0\subset \tilde{[1]})}(M))
\end{array} \longrightarrow
\]
\[ \longrightarrow
\begin{array}{c}
\Delta_{123\*} (j_\*j^!(\mathcal{A}\boxtimes M)) \\ \oplus
\\ \oplus_{\a \in P(\star,1,1)}\  j_{\a\*}j_{\a}^! \Delta_{\a\*}(\mathcal{A}\boxtimes \Gamma_{(I_0\subset \tilde{[1]})}(M)) \\ \oplus
\\ \oplus_{\a \in P(1,\star,\star)} j_{\a\*}j_{\a}^! (\mathcal{A}\boxtimes\Gamma_{(I_0\subset \tilde{[2]})})(M)
\end{array} \longrightarrow \Gamma_{(I_0\subset \tilde{[3]})}(M)\]
where $P(A)$ is the set of cyclic permutations of $A$ and $\Delta_{\a}$ is the inclusion of the diagonal given by $\{x_i = x_j\mbox{ if } \a_i=\a_j\neq \star\}$.  This is a complex because of the Jacobi identity for $\mathcal{A}$
and the Lie action condition for $M$.\\
\\
We can now perform the induction step and construct $M^{\tilde{[n]}}$ for all $n$.  It remains to show that the $M^{\tilde{[n]}}$ do indeed form a factorization module.  We have that $M^{\tilde{[n]}}$ is quasi-isomorphic to the complex
\[ j_\*j^!(\mathcal{A}\boxtimes M^{\widetilde{[n-1]}}) \rightarrow \Gamma_{\left(\widetilde{[n-1]} \subset \widetilde{[n]}\right)}(M^{\widetilde{[n-1]}}) .\]
Suppose we have a diagonal $\Delta: X^{\tilde{[m]}}\rightarrow X^{\tilde{[n]}}$.  If the image of $\Delta$ lies in $H_{\left(\widetilde{[n-1]} \subset \widetilde{[n]}\right)}$, then we have a canonical isomorphism
\[ \Delta^{!} M^{\tilde{[n]}}[n-m] \simeq \Delta^{!} \Gamma_{\left(\widetilde{[n-1]} \subset \widetilde{[n]}\right)}(M^{\widetilde{[n-1]}})[n-m] \simeq M^{\tilde{[m]}} .\]
If the image of $\Delta$ is not in $H_{\left(\widetilde{[n-1]} \subset \widetilde{[n]}\right)}$, we have that $\Delta^{-1}(H_{\left(\widetilde{[n-1]} \subset \widetilde{[n]}\right)}) = H_{\left(\widetilde{[m-1]} \subset \widetilde{[m]}\right)}$.  It follows that
\[ \Delta^{!} M^{\tilde{[n]}}[n-m] \simeq \left( j_\*j^!(\mathcal{A}\boxtimes M^{\widetilde{[m-1]}}) \rightarrow \Gamma_{\left(\widetilde{[m-1]} \subset \widetilde{[m]}\right)}(M^{\widetilde{[m-1]}}) \right)\simeq M^{\tilde{[m]}} .\]
Thus, the $M^{\tilde{[n]}}$ give an object of $\mathcal{M}(Ran_{I_0})$.  Factorization follows from the fact that both $j_\*j^!(\mathcal{A}\boxtimes M^{\widetilde{[n-1]}})$ and $\Gamma_{\left(\widetilde{[n-1]} \subset \widetilde{[n]}\right)}(M^{\widetilde{[n-1]}})$ factorize appropriately.  Furthermore, these isomorphisms are clearly compatible with the unit.
\end{proof}

\section{Lie-* Algebras and Modules}\label{star}
Let $L$ be a Lie-* algebra on $X$.  We will give definitions of Lie-*
and chiral $L$-modules on $X^{I_0}$.  Chiral $L$-modules will be
equivalent to chiral $U(L)$ modules, where $U(L)$ is the chiral
envelope of $L$, and Lie-* modules will be equivalent to modules
over a certain sheaf of Lie algebras.

\subsection{Modules}
We can define chiral and Lie-* modules for a Lie-* algebra analogously to the definition of chiral modules for a chiral algebra.

\begin{defn}
Let $L$ be a Lie-* algebra.
\begin{itemize}
\item
A Lie-* $L$ module on $X^{I_0}$ is a
quasi-coherent sheaf $M$ on $X^{I_0}$ with a map
\[ \mu^{I_0}: L \boxtimes M \rightarrow \Gamma_{(I_0\subset \tilde{[1]})}(M) \]
such that $\mu_{\{1,2\},3} = \mu_{2,\{1,3\}} - \mu_{1,\{2,3\}}$
where
\begin{align*}
\mu_{1,\{2,3\}} &= \mu^{I_0}\circ\mu^{I_0}: L\boxtimes L\boxtimes M \rightarrow L\boxtimes \Gamma_{(I_0\subset \tilde{[1]})}(M)\rightarrow \Gamma_{(I_0\subset \tilde{[2]})}(M), \\
\mu_{2,\{1,3\}} &= \mu_{1,\{2,3\}} \circ \sigma_{12}^! \mbox{ , and} \\
\mu_{\{1,2\},3} &= \mu^{I_0} \circ \mu_L : L\boxtimes L\boxtimes M\rightarrow \Delta_\*(L)\boxtimes M \rightarrow \Delta_\*(\Gamma_{(I_0\subset \tilde{[1]})}(M)) \hookrightarrow \Gamma_{(I_0\subset\tilde{[2]})}(M) .
\end{align*}

\item
A chiral $L$ module on $X^{I_0}$ is a sheaf $M$ with
\[ \mu^{I_0}: j_\*j^!(L\boxtimes M) \rightarrow \Gamma_{(I_0\subset \tilde{[1]})}(M) \]
satisfying the following condition.  For $j': U' \hookrightarrow
X^2\times X^{I_0}$ the complement of the diagonals $\{y_1 = x_i\}$ and
$\{y_2 = x_i \}$, we have (similarly to the case of chiral modules
over a chiral algebra) the maps
\begin{align*}
\mu_{1,\{2,3\}} &: j'_\*j'^!(L\boxtimes L \boxtimes M) \rightarrow j_\*j^!(L\boxtimes L\boxtimes M) \rightarrow \Gamma_{(I_0\subset \tilde{[2]})}(M) \\
 \mu_{2,\{1,3\}} &= \mu_{1,\{2,3\}} \circ \sigma_{12}^!\ \  \mbox{and} \\
 \mu_{\{1,2\},3}&: j_\*'j^{'!}(L\boxtimes L\boxtimes M)\rightarrow \Delta_{12\*}j_\*j^!(L\boxtimes M)\rightarrow \Delta_{12\*} \Gamma_{(I_0\subset \tilde{[1]})}(M) \rightarrow \Gamma_{(I_0\subset \tilde{[2]})}(M) .
\end{align*}
We demand, as before, that
\[ \mu_{\{1,2\},3} = \mu_{2,\{1,3\}} - \mu_{1,\{2,3\}}. \]
\end{itemize}
\end{defn}

\subsection{Topological Algebras}
We define sheaves of topological Lie algebras which give another description of chiral and Lie-* modules for a Lie-* algebra.\\
\\
Consider
\[ H = H_{(I_0\subset \tilde{[1]})} \overset{i}{\longrightarrow} X\times X^{I_0} \overset{p}{\longrightarrow} X^{I_0} .\]
Let $M$ be a $(D_X\boxtimes \mathcal{O}_{X^{I_0}})$-module.  Define
\[ h_{\Gamma}(M) := \underset{\xi\in \Xi}{\mbox{lim }} p_{\*} (M/M_\xi) \]
where $\Xi$ is the projective system of submodules $M_{\xi}\subset M$ such that the quotient $M/M_\xi$ is supported on $H$.  We have that $h_{\Gamma}(M)$ is a pro-quasi-coherent sheaf on $X^{I_0}$, which we can also regard as a sheaf of topological $\mathcal{O}_{X^{I_0}}$-modules by completing.  The functor $h_\Gamma$ is a relative version of the functor $h$ in \cite[\S 2.2]{CHA}.

\begin{lemma}
Let $M$ be a $(D_X\boxtimes \mathcal{O}_{X^{I_0}})$-module and $N,N'$ quasi-coherent sheaves on $X^{I_0}$.  We then have a natural isomorphism
\[ Hom(M\otimes p^!(N)[-1],\Gamma_{(I_0\subset \tilde{[1]})}(N')) \simeq Hom(h_\Gamma(M)\otimes N, N') \]
where the Hom on the right-hand side is taken in the category of pro-quasi-coherent sheaves (or equivalently continuous homomorphisms for $N$ endowed with the discrete topology).
\end{lemma}
\begin{proof}
Both functors
\[ Hom(M\otimes p^{!}(-)[-1],\Gamma_{(I_0\subset \tilde{[1]})}(N')) \mbox{ and } Hom(h_\Gamma(M)\otimes -, N') \]
take colimits to limits.  Therefore, we can assume without loss of generality that $N$ is a vector bundle.  By duality, we have
\[ Hom(M\otimes p^{!}(N)[-1], \Gamma_{(I_0\subset \tilde{[1]})}(N')) \simeq Hom(M, p^!(N^{\vee})[-1]\otimes \Gamma_{(I_0\subset \tilde{[1]})}(N')) \]
Now, since $p^!(N^{\vee})[-1]\otimes \Gamma_{(I_0\subset \tilde{[1]})}(N')$ is supported on $H$, we have
\[ Hom(M, p^!(N^{\vee})[-1]\otimes \Gamma_{(I_0\subset \tilde{[1]})}(N')) \simeq \underset{\xi\in \Xi}{\mbox{colim }} Hom (M/M_{\xi}, p^!(N^{\vee})[-1]\otimes \Gamma_{(I_0\subset \tilde{[1]})}(N')) \simeq \]
\[ \simeq \underset{\xi\in \Xi}{\mbox{colim }} Hom (M/M_{\xi}, \Gamma_{(I_0\subset \tilde{[1]})}(N^{\vee}\otimes N')) \]
Since both sides of the Hom are supported on $H$, we have by Kashiwara's lemma
\[ Hom (M/M_{\xi}, \Gamma_{(I_0\subset \tilde{[1]})}(N^{\vee}\otimes N')) \simeq Hom (i^!(M/M_{\xi}), \pi^!(N^{\vee}\otimes N')) \]
where $\pi: H \rightarrow X^{I_0}$ is the projection map, i.e. $\pi = p\circ i$.  Since $H$ is proper over $X^{I_0}$, we have by adjunction and the projection formula
\[ Hom (i^!(M/M_{\xi}), \pi^!(N^{\vee}\otimes N')) \simeq Hom(p_{\*}(M/M_{\xi}), N^{\vee}\otimes N') .\]
Now, by duality
\[ Hom(p_{\*}(M/M_{\xi}), N^{\vee}\otimes N') \simeq Hom(p_{\*}(M/M_{\xi})\otimes N, N') \]
and therefore we have
\[ Hom(M \otimes p^!(N)[-1], \Gamma_{(I_0\subset \tilde{[1]})}(N')) \simeq \underset{\xi\in \Xi}{\mbox{colim }} Hom(p_{\*}(M/M_{\xi})\otimes N, N') \simeq Hom(h_{\Gamma}(M)\otimes N, N') \]
where the Hom on the right hand side is the set of continuous maps.
\end{proof}

Now, consider
\[ \xymatrix{H_{(I_0\subset \tilde{[1]})} \ar[r]^-{i} & X \times X^{I_0}\ar[dl]_{p_1} \ar[dr]^{p_2} & \ar[l]_-{j} U^{(\tilde{[1]})} \\
X & & X^{I_0} }
\]
\\
For a Lie-* algebra $L$, let
\[ \mathcal{L}_0^{(I_0)}:= h_\Gamma(p_1^!(L)[-I_0]) \ \ \ \mbox{ and } \ \ \ \ \mathcal{L}^{(I_0)}:= h_\Gamma(j_\*j^!(p_1^!(L)[-I_0])). \]
We regard these as sheaves of topological $\mathcal{O}_{X^{I_0}}$-modules.  In fact they each have the structure of a Lie algebra (continuous in each variable) coming from the Lie-* algebra structure on $L$.  In what follows, modules over a sheaf of topological Lie algebras will always mean a quasi-coherent sheaf which is a module for the Lie algebra compatible with the $\mathcal{O}_{X^{I_0}}$-module structure such that the action map is continuous with respect to the discrete topology on the module.\\
\\
Note that at a point $(x_i)\in X^{I_0}$, we have the fibers
\[ \mathcal{L}_{0\ (x_i)}^{(I_0)} = H_{dR}\left(\bigcup_{i\in I_0} D_{x_i}, L\right) \mbox{ and } \mathcal{L}^{(I_0)}_{(x_i)} = H_{dR}\left(\bigcup_{i\in I_0} D^{\times}_{x_i}, L\right) \]
where $D_{x_i}$ (resp. $D^{\times}_{x_i}$) is the formal disk (resp. formal punctured disk) around $x_i \in X$.

\begin{lemma}\label{top}
Let $L$ be a Lie-* algebra.  Then the category of Lie-* (resp. chiral) $L$ modules on $X^{I_0}$ is equivalent to the category of Lie $\mathcal{L}^{(I_0)}_0$ (resp. $\mathcal{L}^{(I_0)}$) modules.
\end{lemma}
\begin{proof}
Let $L'$ be one of $p_1^!(L)[-I_0]$ and $j_\*j^!p_1^!(L)[-I_0]$, and $\mathcal{L} = h_\Gamma(L')$.  For a quasi-coherent sheaf $M$ on $X^{I_0}$, we have natural isomorphisms
\[ Hom(L'\otimes p_2^!(M)[-1], \Gamma_{(I_0\subset \tilde{[1]})}(M)) \simeq Hom(\mathcal{L}\otimes M, M) .\]
It follows that giving a map $L\boxtimes M\rightarrow \Gamma_{(I_0\subset \tilde{[1]})}(M)$ (resp. $j_\*j^!(L\boxtimes M)\rightarrow \Gamma_{(I_0\subset \tilde{[1]})}(M)$) is equivalent to giving a continuous map $\mathcal{L}^{(I_0)}_0 \otimes M \rightarrow M$ (resp. $\mathcal{L}^{(I_0)} \otimes M\rightarrow M$).  Furthermore, the condition of being a module is preserved under this correspondence.
\end{proof}

\subsection{Chiral Envelope}

We will establish the equivalence of chiral $L$ modules and chiral
modules over its chiral envelope $\mathcal{A}(L)$.  The statement is local on $X$, so for the rest of this section assume that $X$ is affine.\\
\\
Recall that the chiral envelope is left adjoint to the functor which
takes a chiral algebra to its underlying Lie-* algebra, i.e. for a
Lie-* algebra $L$ and a chiral algebra $\mathcal{A}$, we have
\[ Hom_{Lie-*}(L,\mathcal{A}) = Hom_{Chiral}(\mathcal{A}(L),\mathcal{A}). \]
Explicitly, the chiral envelope is constructed as a factorization
algebra in \cite{CHA} as follows.
For each $I$, consider
\begin{equation}\label{corr} \xymatrix{H_{(I\subset I\sqcup\{\star\})} \ar[r]^-{i} & X \times X^{I}\ar[dl]_{p_1} \ar[dr]^{p_2} & \ar[l]_-{j} U_I \\
X & & X^{I} }
\end{equation}
where $U_I$ is the complement of $H_{(I\subset I\sqcup\{\star\})}$.  Now let 
\[ \tilde{L}^{(I)} :=
p_{1\*}j_\*j^!p_2^!(L)[-I] \mbox{ and } \tilde{L}^{(I)}_0 :=
p_{1\*}p_2^!(L)[-I] .\]
These are Lie algebras in $D$-modules on $X^I$ and in particular Lie algebras in quasi-coherent sheaves. We
then have $\mathcal{A}(L)^{(I)} = U(\tilde{L}^{(I)})/\tilde{L}^{(I)}_0
U(\tilde{L}^{(I)})$, where $U(-)$ is the universal enveloping algebra of a Lie algebra.\\
\\
As in the case of chiral modules on $X$, we have the following equivalence.

\begin{thm}
Let $L$ be a Lie-* algebra on $X$, $I_0$ a finite set.  As categories over sheaves on $X^{I_0}$ the categories of chiral $L$ modules and chiral $\mathcal{A}(L)$ modules are equivalent.
\end{thm}
\begin{proof}
If $M$ is a chiral $\mathcal{A}(L)$ module, the canonical map $L\rightarrow \mathcal{A}(L)$ makes it into a chiral $L$ module.  Now, suppose that $M$ is a chiral $L$ module on $X^{I_0}$.  By lemma \ref{top}, $M$ is a Lie $\mathcal{L}^{(I_0)}$ module.  In particular, since $\mathcal{L}^{(I_0)}$ is a completion of $\tilde{L}^{(I_0)}$, $M$ is a $\tilde{L}^{(I_0)}$ module.

For each $\tilde{I} = I \sqcup I_0$, consider
\[ X^I \overset{q_1}{\longleftarrow} X^{I}\times X^{I_0}\overset{q_2}{\longrightarrow} X^{I_0} \]
with $q_i$ the projection maps, and let
\[ v: U\hookrightarrow X^{I}\times X^{I_0} \]
be the open set given by $U = \{x_i\neq y_j \ |\  (x_i),(y_j)\in X^{I}\times X^{I_0}\}$. We then have an exact sequence
\begin{equation}\label{topfact}
 0\rightarrow q_2^!(\tilde{L}^{(I)})[-I_0]\rightarrow \tilde{L}^{(\tilde{I})}\rightarrow v_\*v^!q_1^!(L_{(I)})[-I_0] \rightarrow 0
 \end{equation}
where
\[ L_{(I)} := p_{2\*}i_\*i^!p_1^!(L)[1-I] .\]
Let
\[ M^{\tilde{I}} = U(\tilde{L}^{\tilde{I}})\otimes_{q_2^!(\tilde{L}^{(I_0)})[-I]} q_2^{!}(M)[-I]. \]
The $\{M^{\tilde{I}}\}$ factorize because the $\tilde{L}^{(\tilde{I})}$ factorize appropriately by (\ref{topfact}).
\end{proof}

From the theorem, we get an induction functor from Lie-* modules over
$L$ to chiral $L$ modules as follows.  Let $M$ be a Lie-* $L$ module.
It then follows that $M$ is a continuous $\tilde{L}^{(I_0)}_0$ module.
Using induction, we obtain that
\[ \tilde{M} = U(\tilde{L}^{(I_0)})\otimes_{\tilde{L}^{(I_0)}_0} M \]
is a chiral $L$ module.  This is the left adjoint to the forgetful functor from chiral to Lie-* modules.

\subsection{Linear Factorization Sheaves}\label{linfact}
Consider the following diagram
\[ \xymatrix{ H_I \ar[r]^-{i} & \ar[ld]_{p_1} X\times X^{I} \ar[rd]^{p_2} & \\
X & & X^I } \]
where $H_I = H_{(I\subset I\sqcup \{\star\})}$ is the incidence correspondence $\{x=x_i\}$.
For a $D_X$-module $M$, define
\[ M_{(I)}:= p_{2\*}i_\*i^!p_1^!(M)[1-I] .\]
We have that the fiber of $M_{(I)}$ at $(x_i)$ is given by
\[ M_{(I)\ (x_i)} = \underset{x\in \{x_i, i\in I\}}{\bigoplus} M_x \]
where we take the sum over all distinct elements of $\{x_i, i\in I\}$.  Thus the $M_{(I)}$ have a kind of additive factorization property, which we can formalize as follows.

\begin{defn}
A linear factorization sheaf is a $D$-module $\mathcal{F}$ on $Ran(X)$ with compatible factorization isomorphisms
\[ j^{[J/I]!}(\boxplus_{I} \mathcal{F}^{J_i}) \simeq j^{[J/I]!} \mathcal{F}^{J}  \]
for each $J\surj I$.
It is unital if in addition there are given compatible maps
\[ pr_1^!(\mathcal{F}^{(I_1)})[I_1-I] \rightarrow \mathcal{F}^{(I)} \]
for partitions $I=I_1\sqcup I_2$ where $pr_1:X^I\rightarrow X^{I_1}$ is the projection map.
\end{defn}

For a $D_X$-module $M$, we have that the $\{M_{(I)}\}$ form a unital linear factorization sheaf.  The factorization isomorphisms are clear and the unit maps come from the maps of schemes $H_{I_1}\times X^{I_2} \rightarrow H_I$ for a partition $I=I_1\sqcup I_2$.  In fact, these are the only examples of unital linear factorization sheaves.

\begin{thm}
There is an equivalence of categories
\[ \{\mbox{Unital Linear Factorization Sheaves}\}\longrightarrow \{D_X\mbox{-modules}\} \]
given by
\[ \{\mathcal{F}_I\} \mapsto \mathcal{F}_{[1]}
\mbox{ with the inverse }
 M \mapsto \{M_{(I)}\} .\]
\end{thm}
\begin{proof}
We need to show that for a unital linear factorization sheaf $\mathcal{F}$, we have compatible natural isomorphisms
\[ \mathcal{F}_I \simeq \pi_{2\*} \pi_1^{!}(\mathcal{F}_1)[1-I] \]
where $\pi_\alpha = p_\alpha \circ i$ for $\alpha=1,2$ are the projection maps from $H_I$.
By the factorization property, it suffices to construct compatible maps
\[ \pi_{2\*} \pi_1^!(\mathcal{F}_1)[1-I] \rightarrow \mathcal{F}_I \]
The unit provides a map
\[ p_1^!(\mathcal{F}_{[1]})[-I] \rightarrow \mathcal{F}_{(I \sqcup \{\star\})} .\]
Furthermore, by compatibility of the unit map with factorization, we have $i^!(\mathcal{F}_{(I \sqcup \{\star\})})[1]\simeq \pi_2^!(\mathcal{F}_{I})$.
Thus we have compatible maps
\[ \pi_1^!(\mathcal{F}_{[1]})[1-I] \rightarrow \pi_2^!(\mathcal{F}_I) \]
which gives the desired maps by adjunction.
\end{proof}

\subsection{Coaction Map}
Let $L$ be a Lie-* algebra.  We can describe Lie-* $L$ modules on
$X^{I_0}$ as comodules over a certain sheaf of Lie coalgebras.  Since we
are working in the underived setting, we will assume in this section
that $L$ is a vector $D_X$ bundle (in the derived setting it is enough
to assume that $L$ is $D_X$ coherent).\\
\\
In what follows, we will need the following version of duality.
\begin{lemma}\label{duality}
Let $Y_1$, $Y_2$ be smooth schemes and let $M$ be a $D_{Y_1}$-module which is induced from a vector bundle.  Let $N$ be a $\mathcal{O}_{Y_2}$-module and $P$ a $D_{Y_1}\boxtimes \mathcal{O}_{Y_2}$-module.  We then have a canonical isomorphism
\[
Hom(M\boxtimes N, P) \simeq Hom(N, p_{2\*}(p_1^{!}(M^{\vee})[-dim(Y_2)]\otimes P))
\]
where $p_i: Y_1\times Y_2\rightarrow Y_i$ are the projection maps and $M^{\vee}$ is the dual of $M$.
\end{lemma}
\begin{proof}
By duality, we have a natural isomorphism
\[ Hom(M\boxtimes N, P)\simeq Hom(p_2^{\*}(N), p_1^{!}(M^{\vee})[-dim(Y_2)]\otimes P).\]
By adjunction, we further have
\[ Hom(p_2^{!}(N)[-dim(Y_1)], p_1^{!}(M^{\vee})[-dim(Y_2)]\otimes P) \simeq Hom(N, p_{2\*}(p_1^{!}(M^{\vee})[-dim(Y_2)]\otimes P)). \]
\end{proof}

Let $L^\vee$ be the (D-module) dual of $L$.  Then, $L^\vee$ is a Lie
coalgebra in the category of D-modules on $X$.  Let $L^{\vee (I_0)}$ be as in section \ref{linfact}. It is a Lie coalgebra in
$D\mbox{-}mod(X^{I_0})$ and hence a Lie coalgebra in quasi-coherent sheaves on $X^{I_0}$.  We then have the following equivalence.

\begin{propn}
As a category over quasi-coherent sheaves on $X^{I_0}$, the category of $L^{\vee (I_0)}$ comodules is equivalent to the category
of Lie-* $L$ modules on $X^{I_0}$.
\end{propn}
\begin{proof}
Suppose $M$ is an $L^{\vee (I_0)}$ comodule, i.e.
we have a map
\[ coact: M \rightarrow L^{\vee (I_0)} \otimes M. \]
We have, by the projection formula,
\[ L^{\vee (I_0)}\otimes M \simeq p_{2\*}(p_1^!(L^{\vee})[-I_0]\otimes \Gamma_{(I_0\subset \tilde{[1]})}(M)) \]
since $\Gamma_{(I_0\subset \tilde{[1]})}(M)$ is supported on $H$ which is proper over $X^{I_0}$.  By Lemma \ref{duality}, it follows that we have a map
\[ act: L\boxtimes M\rightarrow \Gamma_{(I_0\subset \tilde{[1]})}(M) \]
and so we see that $M$ is a Lie-* $L$ module on $X^{I_0}$ (a similar argument shows that the Lie action identity holds).

The inverse functor is constructed in the same way (every step in the
construction above was invertible).
\end{proof}

\section{$D_X$ Schemes and Commutative Chiral Algebras}\label{spaces}
Let $Y$ be a $D_X$ scheme.  In the case that $Y$ is affine, it is equivalent to a commutative chiral algebra.  In the general case, we consider a factorization space that is defined by $Y$ and give a characterization of factorization spaces that arise in this way.  We also describe modules over a $D_X$ group scheme in terms of Lie-* modules.

\subsection{Multijets}
For our purposes, a space will mean a functor from commutative rings to sets (or even simplicial sets), which we regard as the functor of points.  By Kan extension, this gives a presheaf on the category of schemes.  In the case of a scheme, the corresponding functor is the usual functor of points.  In the case of a stack, we consider the functor of points as taking values in the $\infty$-category of simplicial sets (by identifying groupoids with 1-homotopy types).  A $D_X$-space is a crystal of spaces over $X$.\\
\\
Given a $D_X$-space $Y$, we can construct $D_{X^I}$-spaces $\mathcal{J}_I(Y)$, called multijets, over powers of $X$ as follows.  For a test scheme $S$ an $S$ point of $\mathcal{J}_I(Y)$ is given by a map $\phi: S\rightarrow X^I$ along with a horizontal section $\hat{X}_S\rightarrow Y$ where $\hat{X}_S$ is the completion of $X\times S$ along the subscheme given by the union of the graphs of $\phi$ (regarding it as $I$ maps $S\rightarrow X$).\\
\\
We will give another description of multijets in terms of deRham stacks.  Recall that for a space $Z$, the deRham stack $Z^{dR}$ is given by $Z^{dR}(R) = Z(R/N)$ for a ring $R$ where $N$ is the nilradical of $R$.  A quasi-coherent sheaf on $Z^{dR}$ is exactly the same data as a crystal on $Z$.  Thus, for a smooth variety $Z$, the category of $D$-modules on $Z$ is equivalent to the category of quasi-coherent sheaves on $Z^{dR}$.  More concretely, we have a canonical map
\[ \phi: Z\rightarrow Z^{dR} \] 
and the functor $\phi^*$ implements this equivalence.\\
\\
In the same way, if $Y$ is a $D_Z$-space, then there is a canonical space $Y'$ over $Z^{dR}$ such that
\[ Y\simeq Y'\times_{Z^{dR}} Z .\]

We have the following analogue of induction for $D_X$-spaces.
Given a map of spaces $f: S_1\rightarrow S_2$, we will denote by $f^*$ the functor
\[ f^*:= - \underset{S_2}{\times} S_1: \{\mbox{Spaces}/S_2\} \rightarrow \{\mbox{Spaces}/S_1\} .\]
This functor has a right adjoint $f_*$ called Weil restriction, which is defined as follows.  Let $Z \rightarrow S_1$ be an $S_1$-space.  Then for a ring $R$, we have that $f_*(Z)(R)$ is the set of pairs
\[ (s,\psi) \in \left(S_2(R), Z(Spec(R)\times_{S_2} S_2) \right). \]
\\
Now, consider the following diagram
\[ \xymatrix{ H_I^{dR} \ar[r]^-{i^{dR}} & \ar[ld]_{p_1^{dR}} (X\times X^I)^{dR} \ar[rd]^{p_2^{dR}} & \\
X^{dR} & & (X^I)^{dR} } \]
where $H_I \subset X\times X^I$ is the incidence divisor defined as before.  Let $\pi^{dR}_j = p^{dR}_j\circ i^{dR}$ for $j=1,2$.

\begin{propn}
Let $Y = Y'\times_{X^{dR}} X$ be a $D_X$-space.  Then
\[ \mathcal{J}_I(Y) = (\pi_2^{dR})_*(\pi_1^{dR})^*(Y')\times_{(X^I)^{dR}} X^I. \]
\end{propn}
\begin{proof}
For a scheme $S\rightarrow X^I$, we have that
\[ Hom_{X^I}(S, (\pi_2^{dR})_*(\pi_1^{dR})^*(Y')\times_{(X^I)^{dR}} X^I) = 
Hom_{(X^I)^{dR}}(S, (\pi_2^{dR})_*(\pi_1^{dR})^*(Y')) = \]
\[= Hom_{H_I^{dR}}(S\times_{(X^I)^{dR}} H_I^{dR}, Y\times_{X^{dR}} H_I^{dR}) 
 = Hom_{X^{dR}}(S\times_{(X^I)^{dR}} H_I^{dR}, Y) .\]
To prove the proposition, we need to show that
\[ S\times_{(X^I)^{dR}} H_I^{dR}\times_{X^{dR}} X \simeq \hat{X}_S .\]
In the case that $S=X^I$, we have
\[ X^I\times_{(X^I)^{dR}} H_I^{dR}\times_{X^{dR}} X \simeq (X^I\times X)\times_{(X^I\times X)^{dR}} H_I^{dR} \simeq (X^I\times X)^{\vee}_{H_I} \]
where $(X^I\times X)^{\vee}_{H_I}$ is the formal completion of $X^I\times X$ along $H_I$, which is exactly $\hat{X}_{X^I}$.  Now, for general $S$, we have
\[ S\times_{(X^I)^{dR}} H_I^{dR}\times_{X^{dR}} X \simeq S\times_{X^I}(\hat{X}_{X^I}) \simeq \hat{X}_S. \]
\end{proof}

\begin{rem}
If $Y$ is a $D_X$ scheme then each $\mathcal{J}_I(Y)$ is representable as a $D_{X^I}$ scheme as can be seen, for instance, by the explicit construction of the $D_{X^I}$ scheme $\mathcal{J}_I(Y)$ described in \cite{CHA}.
\end{rem}

\subsection{Factorization Spaces}
\begin{defn}
A factorization space is a space $Z$ over $Ran(X)$ (i.e. a compatible family of spaces $Z^{(I)}$ over $X^I$) together with isomorphisms
\[ j^{[J/I]*}(\times_{I} Z^{(J_i)}) \simeq j^{[J/I]*} Z^{(J)}  \]
for each $J\surj I$
which are mutually compatible.  Furthermore, $Z$ is counital if it comes with a collection of maps
\[ Z^{(I)} \rightarrow X^{I_1} \times Z^{(I_2)} \]
for each partition $I=I_1\sqcup I_2$, which extends the corresponding map over the complement of the diagonal.  We demand that these be isomorphisms over the formal neighborhood of the diagonal.
\end{defn}

Given a $D_X$ space $Y$, the spaces of multijets form a counital factorization space in the obvious way (the counit comes from the maps $X^{I_1}\times H_{I_2} \rightarrow H_{I_1\sqcup I_2}$).

\begin{rem}
In the definition of a counital factorization space $Z$, we demanded that the counit map extend to an isomorphism over the formal neighborhood of the diagonal.  In fact, in the case that $Z$ has a well-behaved cotangent complex (in particular, if $X$ is a scheme or an algebraic stack) this extra condition is equivalent to $Z$ being a derived space over $Ran(X)$, i.e. that we have a compatible family of isomorphisms
\[ Z^{(J)} \times^L_{X^J} X^I \simeq Z^{(I)} \]
for every surjection $J\surj I$.  In particular if the $Z^{(I)}$ are algebraic stacks, we have that $Z^{(J)} \times_{X^J} X^I \simeq Z^{(J)} \times^L_{X^J} X^I$.
\end{rem}

Now, suppose that $\{Z^{(I)}\}$ is a counital factorization space.  Then each $Z^{(I)}$ acquires a canonical connection. Namely, on $X^{I\sqcup I}$, we have maps
\[ X^I\times Z^{(I)} \leftarrow Z^{(I\sqcup I)} \rightarrow Z^{(I)}\times X^I \]
which are isomorphisms over the formal neighborhood of the diagonal.  This is exactly the data of a connection on $Z^{(I)}$.  These connections are compatible with the factorization structure.

\begin{thm}\label{factequiv}
The functor
\[ \{D_X\mbox{ Spaces}\} \longrightarrow \{\mbox{Counital Factorization Spaces}\} \]
given by
$ Y \mapsto \{\mathcal{J}_I(Y)\} $
is right adjoint to the functor $\{Z^{(I)}\}\mapsto Z^{([1])}$.
\end{thm}
\begin{proof}
Let $Y = Y'\times_{X^{dR}} X$ be a $D_X$-space and $\{Z^{(I)}\}$ a factorization space.
Since each $Z^{(I)}$ has a canonical connection along $X^I$, we have that $Z^{(I)} = Z'^{(I)}\times_{(X^I)^{dR}} X^I$.  Because the connection is compatible with factorization, the $Z'^{(I)}$ form a factorization space over $X^{dR}$.  Now, suppose we have a map $Z'^{([1])} \rightarrow Y'$ (which is equivalent to a map of the corresponding $D_X$-spaces).  The counit for $Z$ gives
\[ Z'^{(I\sqcup [1])} \rightarrow p_2^{dR *} Z'^{(I)} \]
which is an isomorphism over $H_I^{dR}$.  The counit also gives a map
\[ Z'^{(I\sqcup [1])} \rightarrow p_1^{dR *}(Z'^{([1])}) .\]
Restricting to $H_I^{dR}$, we get
\[ i^{dR *}(Z'^{(I\sqcup [1])})\simeq \pi_2^{dR *}(Z'^{(I)})\rightarrow \pi_1^{dR *}(Z'^{([1])}) .\]
It follows that the map $Z'^{([1])}\rightarrow Y'$ gives maps
\[ \pi_2^{dR *}(Z'^{(I)})\rightarrow \pi_1^{dR *} (Z'^{([1])}) \rightarrow \pi_1^{dR *}(Y') .\]
By adjunction, we have
\[ Z'^{(I)}\rightarrow \pi_{2*}^{dR} \pi_1^{dR *}(Y').\]
Pulling back to $X^I$, we obtain the maps
\[ Z^{(I)} \rightarrow \mathcal{J}_I(Y) \]
which give a morphism of factorization spaces.  In fact, every morphism $\{Z^{I}\}\rightarrow \{\mathcal{J}_I(Y)\}$ is obtained in this way by the same argument.
\end{proof}

Note that for a $D_X$ space $Z$, we have that $Z^{([1])} \simeq Z$.  It follows that the functor
\[ \{D_X\mbox{ Spaces}\} \longrightarrow \{\mbox{Counital Factorization Spaces}\} \]
is fully faithful.

\begin{cor} Suppose that $Z$ is a counital factorization space with each $Z^{(I)}$ an algebraic stack with affine diagonal admitting an fpqc cover by a scheme.  Then $Z$ is the space of multijets of $Z^{([1])}$.
\end{cor}
\begin{proof}
By Theorem \ref{factequiv}, we have a map
\[ \phi: Z^{(I)} \rightarrow \mathcal{J}_I(Z^{([1])}) .\]
Let $H_n \subset X^I$ be the union of diagonals of codimension $> n$.  We will show by induction that $\phi$ is an isomorphism over $X^I-H_n$, and in particular is an isomorphism over $X^I$.  By construction, $\phi$ is an isomorphism over $X^I-H_0$. Now, suppose that $\phi$ is an isomorphism over $X^I-H_{n-1}$. To simplify notation, let
\[ \phi_n: Z_n:=Z^{(I)}\times_{X^I} X^I-H_{n-1} \rightarrow Z'_n:=\mathcal{J}_I(Z^{([1])})\times_{X^I} X^I-H_{n-1} .\]
We have, by induction and construction, that $\phi_n$ is an isomorphism over $X^I-H_{n-1}$ and its complement $H_{n-1}-H_{n} = V(f_n)\subset X^I-H_{n}$. Now for $\mathcal{F} \in DCoh(Z'_n)$ a complex of quasi-coherent sheaves on $Z'_n$, we have that the unit map
\[ \eta: \mathcal{F}\rightarrow \phi_{n *}\phi_n^*(\mathcal{F}) \]
is an isomorphism when restricted to $Z'_{n-1}$.  Let $v': Z'_n-Z'_{n-1}\hookrightarrow Z'_n$ be the inclusion.  Since $v'_*v'^*: DCoh(Z'_n)\rightarrow DCoh(Z'_n)$ is given by tensoring with a perfect complex (and similarly for $Z_n$), we have that $\eta$ is an isomorphism over $Z'_n-Z'_{n-1}$.  It follows that $\eta$ is an isomorphism.  Similarly, the counit map
\[ \phi_n^*\phi_{n *} \mathcal{F} \rightarrow \mathcal{F} \]
is also an isomorphism.
Thus, $\phi_{n *}$ and $\phi_n^*$ are mutually inverse equivalences of categories.\\
\\
Let's show that $\phi_n^*$ is $t$-exact for the natural $t$-structures on the two derived categories.  Clearly, $\phi_n^*$ preserves connective objects (in the homological grading convention).  Now, suppose that $\mathcal{F} \in DCoh(Z')$ is coconnective, i.e. $H^i(\mathcal{F})=0$ for $i<0$.  We have an exact triangle
\[ \mathcal{F} \overset{f_n}{\rightarrow} \mathcal{F} \rightarrow v'_*v'^*\mathcal{F} \]
It follows that $H^i(v'_*v'^*(\mathcal{F}))=0$ for $i<-1$.  Now, since $\phi_n^*(\mathcal{F})$ is coconnective when restricted to $Z_{n-1}$, we have that for $i<0$ and for each (local) section $x \in H^i(\phi_n^*(\mathcal{F}))$, $f_n^k x = 0$ for some $k\geq 0$.  Now suppose there exists a nonzero $x\in H^i(\phi_n^*(\mathcal{F}))$ for $i<0$.  Without loss of generality, we can assume that $f_nx = 0$.  The exact triangle
\[ \phi_n^*(\mathcal{F})\overset{f_n}{\rightarrow}\phi_n^*(\mathcal{F})\rightarrow v_*v^*\phi_n^*(\mathcal{F}) \]
(where $v: Z_n-Z_{n-1} \hookrightarrow Z_n$) gives a nonzero class in $H^{i-1}(v_*v^*(\mathcal{F})) \simeq H^{i-1}(v'_*v'^*(\mathcal{F}))$ = 0, a contradiction.  It follows that $\phi_n^*$ induces an equivalence of the abelian tensor categories $QCoh(Z)$ and $QCoh(Z')$, which by tracing the argument of \cite{tannaka} implies that $\phi_n$ is an isomorphism.
\end{proof}

\subsection{$D_X$ Group Schemes}
Let $G$ be an affine $D_X$ group scheme, and
$\mathcal{O}_G$ the algebra of functions.  In this case, $\mathcal{O}_G$ naturally has the structure of a commutative chiral Hopf algebra and we obtain a factorization group $D_X$ scheme, i.e. for each power of the curve $X^{I}$, we have an
affine $D_X$ group scheme $G^{(I)}$.

\begin{defn}
Let $G$ be an affine $D_X$ group scheme, $I_0$ a finite set.  A $G$
module on $X^{I_0}$ is a quasi-coherent sheaf on $X^{I_0}$ which is a comodule
for the Hopf algebra $\mathcal{O}_{G^{(I_0)}}$.
\end{defn}

Now, let $G$ be a smooth affine $D_X$ group scheme.
Let $\Omega^1_{G^{(I_0)}, e}$ be the sheaf of differentials of $G^{(I_0)}$
(as a scheme over $X^{I_0}$) pulled back along the identity section.
Since $G^{(I_0)}$ is naturally a $D_{X^{I_0}}$ scheme,
$\Omega^1_{G^{(I_0)},e}$ acquires the structure of a $D_{X^{I_0}}$ module
compatible with the Lie coalgebra structure.

Since $\Omega^1_{G,e} = L^{\vee}$, where $L$ is the tangent space of $G$ at the identity, $L$ naturally acquires the structure of a Lie-* algebra.

As for any affine group
scheme, if $M$ is a module over $G^{(I_0)}$, it becomes a comodule over
the Lie-coalgebra $\Omega^1_{G^{(I_0)},e}$.  The goal of this section is to construct a functor from $G$-modules to Lie-* $L$-modules.  In particular, we will express $\Omega^1_{G^{(I_0)},e}$ in terms of $L$.

\begin{propn}
Let $G$ be an affine $D_X$ group scheme with Lie-* algebra $L$.  Then the Lie coalgebra of differentials at the identity of $G^{(I_0)}$ is isomorphic to $L^{\vee (I_0)}$.
\end{propn}
\begin{proof}
Since $\{G^{(I_0)}\}$ is a counital factorization space, we have that $\{\Omega^1_{G^{(I_0)},e}\}$ is a unital linear factorization sheaf.  Furthermore, the coalgebra structure is compatible with the factorization structure.  Thus since $\Omega^1_{G,e} = L^{\vee}$, we have $\Omega^1_{G^{(I_0)},e} = L^{\vee (I_0)}$.
\end{proof}

\begin{cor}
For a $D_X$ group scheme $G$, there is a functor
\[ \{ G^{(I_0)}\mbox{-modules} \} \rightarrow \{\mbox{Lie-* $L$-modules on $X^{I_0}$} \} \]
which is fully faithful when $G$ is connected.
\end{cor}

\section{Tensor Categories and Modules}\label{tensor}
Recall from \cite{CHA}, that we can embed the pseudo tensor category
of sheaves on the Ran space into a larger tensor category
$\mathcal{M}(X^\mathcal{S})$ (for both the $*$ and the chiral pseudo
tensor structures). We will describe a similar picture for chiral and
Lie-* modules.  In this section we work primarily in the derived setting.

\subsection{Tensor Structures}
Recall from \cite{CHA} that $\mathcal{M}(X^{\mathcal{S}})$ has two tensor structures so that the embedding
\[ \mathcal{M}(X) \rightarrow \mathcal{M}(X^{\mathcal{S}}) \]
is an embedding of pseudo-tensor categories with respect to both the $*$ and the chiral pseudo-tensor structures.  Furthermore, chiral algebras and Lie-* algebras are Lie algebras with respect to the chiral and *-(pseudo-)tensor products respectively.  We describe a similar picture for $D\mathcal{M}(X^{I_0})$ and $D\mathcal{M}(Ran_{I_0}(X))$.

The category $\mathcal{M}(X^{I_0})$ is a pseudo-tensor module category for the pseudo-tensor category $\mathcal{M}(X)$ (for both structures).  Namely, for a collection of objects $M_{i} \in \mathcal{M}(X)$ for $i\in I$ and $N_1,N_2\in\mathcal{M}(X^{I_0})$ we have
\[ Hom_I^*(\{M_i; N_1\}, N_2) = Hom_{\mathcal{M}(X^{\tilde{I}})}\left(\underset{i\in I}{\boxtimes} M_i \boxtimes N_1, \Gamma_{(I_0\subset \tilde{I})}(N_2)\right) \]
and
\[ Hom_I^{ch}(\{M_i; N_1\}, N_2) = Hom_{\mathcal{M}(X^{\tilde{I}})}\left(j_\*j^!\left(\underset{i\in I}{\boxtimes} M_i \boxtimes N_1\right), \Gamma_{(I_0\subset \tilde{I})}(N_2)\right) \]
where $j: U^{(\tilde{I})}\rightarrow X^{\tilde{I}}$ is the inclusion of the complement of the diagonals.  We have that a chiral (resp. Lie-*) module over a chiral (resp. Lie-*) algebra $\mathcal{A}$ on $X^{I_0}$ is exactly a Lie module in the pseudo-tensor module category $\mathcal{M}(X^{I_0})$ for the Lie algebra $\mathcal{A}$ in the pseudo-tensor category $\mathcal{M}(X)$ (with the appropriate pseduo-tensor structure).

The preceding discussion works just as well in the derived setting, i.e. $D\mathcal{M}(X)$ is has two pseudo tensor structures and $D\mathcal{M}(X^{I_0})$ is a module category over it for both of them.  The embedding $D\mathcal{M}(X)\rightarrow D\mathcal{M}(X^{S})$ is a fully faithful embedding of pseduo-tensor categories (for both pseduo-tensor structures) where the tensor products in $D\mathcal{M}(X^{S})$ are given by
\[ (\otimes^* M_i)_{X^J} := \underset{J\surj I}{\oplus} \boxtimes_I (M_i)_{X^{J_i}} \]
and
\[ (\otimes^{ch} M_i)_{X^J} := \underset{J\surj I}{\oplus}\  j_\*^{(J/I)}j^{(J/I) !} \boxtimes_I (M_i)_{X^{J_i}} \]
for objects $M_i$ in $D\mathcal{M}(X^\mathcal{S})$, $i\in I$.  \\
\\
For a finite set $I_0$, the category $D\mathcal{M}(X^{\mathcal{S}_{I_0}})$ is a module category over $D\mathcal{M}(X^{\mathcal{S}})$ with respect to both tensor structures.
Namely, for $M \in D\mathcal{M}(X^\mathcal{S})$ and $N\in
D\mathcal{M}(X^{\mathcal{S}_k})$,
\[ (M\otimes^* N )_{X^{\tilde{J}}} := \underset{\tilde{J}\surj [1]^+}{\oplus} M_{X^{J_1}} \boxtimes N_{X^{J_\star}}\]
and
\[ (M\otimes^{ch} N )_{X^{\tilde{J}}} := \underset{\tilde{J}\surj [1]^+}{\oplus} j_\*^{(\tilde{J}/[1]^+)}j^{(\tilde{J}/[1]^+) !}M_{X^{J_1}} \boxtimes N_{X^{J_\star}} .\]

Recall that we have adjoint functors
\[ r_*: D\mathcal{M}(X^{\mathcal{S}_{I_0}}) \lrarrow D\mathcal{M}(Ran_{I_0}) : r^! \]
with $r^!$ fully faithful.  In the case that $I_0$ is the empty set, we can define the chiral and *-tensor products on $D\mathcal{M}(Ran(X))$; for $M_1,M_2 \in D\mathcal{M}(Ran)$ set
\[ M_1 \otimes^* M_2 = r_* (r^!(M_1)\otimes^*r^!(M_2)) \]
and similarly for the chiral tensor product.  We have that both $r^!$ and $r_*$ are tensor functors.
In the same way, we define the action of $D\mathcal{M}(Ran)$ on $D\mathcal{M}(Ran_{I_0})$.  Our goal is to prove

\begin{thm}\label{pt}
The functor
\[ r_{I_0*}: D\mathcal{M}(X^{I_0}) \rightarrow D\mathcal{M}(Ran_{I_0}) \]
is a fully faithful embedding of pseudo-tensor module $D\mathcal{M}(X)$ categories (for both pseudo-tensor structures).
\end{thm}

\begin{cor}\label{lieequiv}
Let $L$ be a Lie-* (resp. chiral) algebra, $M$ a sheaf on $X^{I_0}$.  Then
$M$ is a Lie-* (resp. nonunital chiral) $L$ module if an only if
$r_{I_0 *}(M)$ is a Lie $r_{*}(L)$ module
for the $\otimes^*$ (resp. $\otimes^{ch}$) tensor structure.  In the case of a chiral module, $M$ is unital if in addition the restriction of $M$ to $r_*(\omega)$ is the canonical module structure.
\end{cor}

\begin{lemma}\label{ptensor}
Let $M\in D\mathcal{M}(X^{I})$ and $N\in D\mathcal{M}(X^{\tilde{J}})$.  Then
\[ s_{I\sqcup \tilde{J}*}\left(M \boxtimes N\right) = s_{I*} (M) \otimes^* s_{\tilde{J} *}(N) \]
and
\[ s_{I\sqcup \tilde{J}*}\left(j_{\*}j^!(M \boxtimes N)\right) = s_{I*} (M) \otimes^{ch} s_{\tilde{J} *}(N) \]
where $j: U\hookrightarrow X^I\times X^J$ is the open set given by $\{x_i\neq y_j\}$ with $x_i$ and $y_j$ coordinates on $X^I$ and $X^J$ respectively.
\end{lemma}
\begin{proof}
We have
\[ \left(s_{I *}(M) \otimes^* s_{\tilde{J} *}(N)\right)_{X^{\tilde{K}}}
= \bigoplus_{\pi: \tilde{K}\surj [1]^{+}} \left(s_{I *}(M)\right)_{K_1} \boxtimes \left(s_{\tilde{J} *}(N)\right)_{K_{\star}} = \]
\[
= \bigoplus_{\pi: \tilde{K}\surj [1]^{+}} \left( \bigoplus_{\pi_1: K_1\surj I} \Delta^{(\pi_1)}_{\*} M \right) \boxtimes \left(\bigoplus_{\pi_2: K_\star \surj \tilde{J}} \Delta^{(\pi_2)}_{\*} N \right) = \]
\[ = \bigoplus_{\tilde{K} \surj I\sqcup \tilde{J}} \Delta_{\*}^{(\pi)} M\boxtimes N 
= s_{I\sqcup \tilde{J} *} \left(M \boxtimes N\right)_{X^{\tilde{K}}}. \]
The proof for the chiral tensor product is the same.
\end{proof}

\begin{proof}[Proof of Theorem \ref{pt}]
We will give the proof for the *-pseudo-tensor structure.  The proof for the chiral pseudo-tensor structure is the same.  We need to show that for $M_i \in D\mathcal{M}(X)$, $i\in I$ and $N_1,N_2 \in D\mathcal{M}(X^{I_0})$ we have natural isomorphisms
\[ Hom_I^*(\{M_i ; N_1\}, N_2) \simeq Hom_I^*(\{M_i ; r_{I_0 *}(N_1)\}, r_{I_0 *}(N_2))\]
By definition,
\[ Hom_I^*(\{M_i ; r_{I_0 *}(N_1)\}, r_{I_0 *}(N_2)) = Hom_{D\mathcal{M}(Ran_{I_0})}\left(\bigotimes^*_i r_{[1] *}(M_i) \otimes^* r_{I_0 *}(N_1), r_{I_0 *}(N_2)\right).\]

Recall that the functors
\[ r_{\tilde{I}*} : D\mathcal{M}(X^{\tilde{I}})\rightarrow D\mathcal{M}(Ran_{I_0}) \]
factor as the composition $r_{\tilde{I}*} = r_* \circ s_{\tilde{I} *}$.
Therefore, by Lemma \ref{ptensor}, we have that
\[ \bigotimes^*_i r_{[1] *}(M_i) \otimes^* r_{I_0 *}(N_1) \simeq r_{\tilde{I} *}(\boxtimes_i M_i\boxtimes N_1).\]
Furthermore, by adjunction
\[ Hom_{D\mathcal{M}(Ran_{I_0})}(r_{\tilde{I} *}(\boxtimes_i M_i\boxtimes N_1), r_{I_0*}(N_2)) \simeq Hom_{D\mathcal{M}(X^{\tilde{I}})}(\boxtimes_i M_i\boxtimes N_1, r_{\tilde{I}}^!r_{I_0*}(N_2)) .\]
We have that $r_{\tilde{I}}^!r_{I_0*}(N_2) \simeq \Gamma_{I_0\subset \tilde{I}}(N_2)$, which gives the desired isomorphisms.
\end{proof}

\subsection{Chevalley complex}
Let $\mathcal{A}$ be a chiral algebra, $M$ a chiral module
$\mathcal{A}$ module on $X^{I_0}$.  We then have that
$\tilde{A}:=r_*(\mathcal{A})$ is a Lie algebra (with respect to
the $\otimes^{ch}$ tensor structre) and $\tilde{M}:=r_{I_0 *}(M)$ is a $\tilde{A}$-
module.  Therefore, we can form the Chevalley complex in $D\mathcal{M}(Ran_{I_0})$:
\[ C_*(\tilde{A},\tilde{M}) := \ldots \darrow C_n=Sym^{n}(\tilde{A}[1])\otimes^{ch} \tilde{M} \darrow C_{n-1}\ldots \]
where, by convention, $C_0 = Sym^0(\tilde{A}[1])\otimes \tilde{M} =
\tilde{M}$.

\begin{propn}\label{chev}
Let $\mathcal{A}$ be a chiral algebra, $M$ a factorization $\mathcal{A}$
module on $X^{I_0}$.  Then, as an object of $D\mathcal{M}(Ran_{I_0})$, $M$ is quasi-isomorphic
to the Chevalley complex $C_*(r_*(\mathcal{A}),r_{I_0 *}(M))$.
\end{propn}
\begin{proof}
On $X^{\tilde{J}}$, we have
\[ (Sym^{n}(\tilde{A}) \otimes^{ch} \tilde{M})_{X^{\tilde{J}}} = 
\underset{\tilde{J}\surj [1]^+}{\oplus} j_\*^{(\tilde{J}/[1]^+)}j^{(\tilde{J}/[1]^+) !} (Sym^n(\tilde{A})_{X^{J_1}} \boxtimes \Gamma_{(I_0\subset J_\star)}(M)) = \]
\[ =
\underset{\tilde{J}\surj [1]^+}{\oplus} \underset{J_1\surj [n]}{\oplus} j_\*^{(\tilde{J}/[1]^+)}j^{(\tilde{J}/[1]^+) !}( j^{(J_1/[n])}_\*j^{(J_1/[n]) !} (\boxtimes_{i\in [n]} \Delta_{\*}^{(J_1)_i} \mathcal{A} \boxtimes \Gamma_{(I_0\subset J_\star)}(M))_{\Sigma_n} ) = \] 
\[ =
\left(\underset{\tilde{J}\surj [n]^+}{\oplus} j_\*^{(\tilde{J}/[n]^+)}j^{(\tilde{J}/[n]^+) !} (\boxtimes_{i\in [n]} \Delta_{\*}^{J_i} \mathcal{A} \boxtimes \Gamma_{(I_0\subset J_\star)}(M) ) \right)_{\Sigma_n}
 \]
with the differentials between these given by the various
multiplication maps.  Thus we see that this Chevalley complex is
exactly the Cousin complex for $M^{(\tilde{J})}$ along the
stratification given by the diagonals.
\end{proof}

We can use this fact to give a description of modules which are in the essential image of $r_{I_0 *}$ in $\mathcal{M}(Ran_{I_0})$.

\begin{propn}\label{liefact}
Let $\mathcal{A}$ be a chiral algebra, and $M\in \mathcal{M}(Ran_{I_0})$ a Lie module for $r_*(\mathcal{A})$.  Then the Chevalley complex $C_*(r_*(\mathcal{A}), M)\in D\mathcal{M}(Ran_{I_0})$ is a factorization module for $\mathcal{A}$ if and only if $M$ is in the essential image of
\[ r_{I_0 *}: \{\mbox{Chiral } \mathcal{A}\mbox{-mods}\}\rightarrow \{\mbox{Lie } r_*(\mathcal{A})\mbox{-mods}\}. \]
\end{propn}
\begin{proof}
By the above, if $M = r_{I_0 *}(M')$ then its Chevalley complex is a factorization module.

Now, suppose $C(r_*(\mathcal{A}), M)$ is a factorization module for $\mathcal{A}$.  
This factorization module is equivalent to the chiral module $M_{X^{I_0}}$ and the inverse equivalence is given by the Cousin complex for the stratification along the diagonals, which by Proposition \ref{chev} is exactly the Chevalley complex $C(r_*(\mathcal{A}), r_{I_0 *}(M_{X^{I_0}}))$.  It follows that
\[ C(r_*(\mathcal{A}), M) \simeq C(r_*(\mathcal{A}), r_{I_0 *}(M_{X^{I_0}})).\]
Let $N$ be the cone of the natural map $r_{I_0 *}(M_{X^{I_0}})\rightarrow M$.  It has a natural structure of a Lie $r_{I_0 *}(\mathcal{A})$ module, and we have $C(r_*(\mathcal{A}), N)\simeq 0$.  To prove that $M\simeq r_{I_0 *}(M_{X^{I_0}})$, we need to show that $N\simeq 0$.  We will do this by showing that $N_{X^{\tilde{I}}}\simeq 0$ for all finite sets $I$ by induction on $|I|$.  For $|I|<0$ there is nothing to prove.  Now suppose $N_{X^{\tilde{I}}}\simeq 0$ for $|I|<|J|$.  It follows that $C(r_*(\mathcal{A}), N)_{X^{\tilde{J}}}\simeq N_{X^{\tilde{J}}}$ and in particular, $N_{X^{\tilde{J}}}\simeq 0$.
\end{proof}

\begin{rem}
J. Francis \cite{JNKF} previously gave a proof of Proposition \ref{liefact}, which does not appeal to the equivalence of chiral and factorization modules, using different methods.  Thus, we can give an alternate proof of Theorem \ref{equiv} using Corollary \ref{lieequiv} and Proposition \ref{liefact}.\end{rem}

\end{document}